\documentclass[12pt]{article}
\usepackage{amsmath}
\usepackage{amssymb}
\usepackage{amsthm}
\usepackage{epsfig} 
\usepackage{wrapfig}
\usepackage{pdfpages}
\usepackage[pdftex]{hyperref}
\newcommand {\R}{\mathbb R}

\newcommand {\ct}{{\cal T}}
\newcommand{\p}{\mathbb P}
\newcommand{\php}{(\Phi,\p)}
\newcommand{\Om}{\Omega}
\newcommand{\e}{\mathbb E}
\newcommand{\dd}{\Delta}

\newcommand {\eq}[1]{\begin {equation} #1 \end {equation}}
\newcommand {\eqn}[1]{\begin {equation*} #1 \end {equation*}}

\newcommand {\f}{F}
\newcommand {\h}{H}
\newcommand {\ff}{g}
\newcommand {\g}{f}
\newcommand{\rf}{\f}

\newcommand\lk{\left[}
\newcommand\rk{\right]}

\newcommand\lp{\left(}
\newcommand\rp{\right)}
\newcommand{\ep}{\e_\Phi}

\renewcommand {\sf}{{\cal F}}
\theoremstyle{definition}\newtheorem{thm}{Theorem}
\theoremstyle{definition}\newtheorem{lem}[thm]{Lemma}
\theoremstyle{definition}\newtheorem{cor}[thm]{Corollary}
\theoremstyle{definition}\newtheorem{exm}[thm]{Example}
\theoremstyle{definition}\newtheorem{defi}[thm]{Definition}
\theoremstyle{definition}\newtheorem{rem}[thm]{Remark}
\theoremstyle{definition}\newtheorem{prop}[thm]{Proposition}
\renewcommand {\sp}{\R^d}

\newcommand{\norm}[1]{|{#1}|}
\newcommand{\n}{\mathbf N}
\newcommand{\N}{{\cal N}}

\newcommand{\cp}{{\cal P}}
\newcommand{\B}{{\cal B}}

\newcommand{\fp}{{\p_\Phi}}

\newcommand{\Z}{\mathbb Z}

\newcommand{\ind}{\mathbf 1}

\newcommand{\fo}{\f_\bot}
\newcommand{\rfo}{f_\bot}

\newcommand{\rfp}{\rf_\phi}

\newcommand{\fe}{\e_\Phi}

\newcommand{\sif}{\sim_{\ff}}
\newcommand{\pom}{point-shift}
\newcommand{\Pom}{Point-shift}
\newcommand{\PoM}{Point-Shift}
\newcommand{\pomk}{{point-map}}
\newcommand{\PoMk}{{Point-Map}}
\newcommand{\Pomk}{{Point-map}}
\renewcommand{\l}[2][]{L^{#2}_{#1}}
\renewcommand{\c}[2][]{C^{#2}_{#1}}
\renewcommand{\L}[2][]{{\cal L}^{#2}_{#1}}
\newcommand{\C}[2][]{{\cal C}^{#2}_{#1}}
\newcommand{\thetg}{{\theta_{\g}}}

\newcommand{\Ltg}{\L{\thetg}}
\newcommand{\Ctg}{\C{\thetg}}
\newcommand{\om}{\omega}
\renewcommand{\Z}{\Bbb Z}
\newcommand{\asy}{almost surely}

\newcommand{\fph}{F_\phi}

\newcommand{\theog}{\theta_{\g_\bot}}

\newcommand{\ffsubph}{F_\phi}
\newcommand{\fst}{f_{s}}

\begin{document}

\author{F. Baccelli\thanks{baccelli@math.utexas.edu} \\{\small University of Texas at Austin}
\and M.-O. Haji-Mirsadeghi\thanks{mirsadeghi@sharif.ir}\\{\small Sharif University of Technology}}

\title{Point-Shift Foliation of a Point Process}

\date{}

\maketitle

\begin{abstract}
A point-shift $F$ maps each point of a point process $\Phi$ to
some point of $\Phi$.
For all translation invariant point-shifts $F$,
the $F$-foliation of $\Phi$ is a partition of the support
of $\Phi$ which is the discrete analogue of the stable manifold of
$F$ on $\Phi$.
It is first shown that foliations lead to a classification
of the behavior of point-shifts on point processes.
Both qualitative and quantitative properties of foliations are
then established. It is shown that for all point-shifts $F$,
there exists a point-shift $F_\bot$, the orbits of which are
the $F$-foils of $\Phi$, and which is measure-preserving.
The foils are not always stationary point processes.
Nevertheless, they admit relative intensities with respect
to one another.
\end{abstract}



\tableofcontents

\section{Introduction}
A point process is said to be flow-adapted 
if its distribution is invariant by the group of translations on $\R^d$.
A point-shift is a dynamics on the support
of a flow-adapted point process, which is itself flow-adapted.

The main new objects of the paper are the notion
of foliation of a flow-adapted point process w.r.t.
a flow-adapted point-shift.

Such a foliation
is a discrete version of the global stable manifold (see e.g. \cite{KaHa95}
for the general setting and below for the precise definition
used here) of this dynamics, i.e.,
two points in the support of the point process are in the
same {\em leave} or {\em foil} of this stable manifold
if they have the same ``long term behavior'' for this dynamics.
This foliation provides a flow-adapted partition of the support
of the point process in connected components and foils.

The {\em point foil} of a point process w.r.t. a point-shift
is defined under the Palm distribution of this point
process. It is the random counting measure with atoms at
the points of the foil of the origin. 
The distribution of the point foil 
under the Palm probability of the point process
is left invariant by all bijective shifts preserving the foliation.
A point foil is not always {\em markable}, i.e., is not always
a stationary point process under its Palm distribution. 

The main mathematical result of the paper is the classification 
of point-shifts based on the cardinalities of their foils and
connected components (Theorem \ref{cardinality.components})
and on whether their point foils are markable or not.

The literature on point-shifts
starts with the seminal paper by
J. Mecke \cite{Me75}. The fundamental result of \cite{Me75} is
the {\em point stationarity theorem}, which states that all bijective
point-shifts preserve the Palm distribution of all simple and
stationary point processes.
The notion of {\em point-map} was introduced by
H. Thorisson (see \cite{Th00} and the references therein) and
further studied by M. Heveling and G. Last \cite{HeLa05}.
The dynamical system analysis of point-shifts which is pursued 
in the present paper was proposed in \cite{BaHaM}.   
The last paper is focused on long term properties of iterates of
point-shifts. It introduces the notion of {\em point-map probability},
which provides an extension of Mecke's point stationarity theorem.
In contrast, the present paper is focused on the stable manifold
of a point-shift, as already mentioned.
It is centered on the definition of this object and 
on the study of
both qualitative and quantitative properties of its distribution.

The paper is structured as follows. 
Section \ref{sec:difo} defines the setting for discrete foliations
and Section \ref{sec:ppps} that for point processes and point-shifts.
Section \ref{sec:psfo} combines the two frameworks and defines
the discrete foliation of a point process by a point-shift.
Section \ref{sec:cacl} gives the classification.
Section \ref{sec:pppm}
introduces the stable group of this foliation, and shows
the existence of measure preserving dynamics on the foliation. It also 
defines the foil point process.
Finally, Section \ref{sec:quant} gathers the quantitative
properties of foliations.

\section{Discrete Foliations}
\label{sec:difo}

\subsection{Foils and Connected Components}
The notion of discrete foliation can be defined for any
function on any set. Since the present paper is focused on stochastic objects,
only measurable functions on measurable spaces will be considered. 

Assume $(X,\sf)$ is a measurable space where all
singletons are measurable;
i.e., for all $x\in X$ one has $\{x\}\in\sf$
and let $\ff$ be a measurable map (or dynamics) on $X$\footnote{  
When $X$ is a topological space
and $\ff$ is continuous, $\ff$ defines a topological dynamical
system; when $X$ is equipped with a probability
measure which is preserved by $\ff$, the latter defines a measure
preserving dynamical system.}.
Let $\sif$ be the binary relation on the elements of $X$ defined by 
\eqn{x\sif y\Leftrightarrow \exists n\in \mathbb N;\ff^n(x) = \ff^n(y).}
It is immediate that $\sif$ is an equivalence relation. 
\begin{defi}\label{foliation}
The partition of $X$ generated by the equivalence
classes of $\sif$ will be called the \emph{$\ff$-foliation of $X$}.
Denote it by $\L \ff(X)$ or $\L [X]\ff$.
Each equivalence class is called a \emph{foil}.
The equivalence class of $x\in X$
is denoted by $\l \ff(x)$.
\end{defi}

\begin{rem}
In the terminology of geometry, foils are called leaves.
But since the paper uses graphs which are mostly trees,
to avoid confusion with tree leaves,
the word foil will be used here. 
\end{rem}

One can also see $\l \ff(x)$ as the limit of the increasing
sets $\l[n] \ff(x)$, where 
$$\l[n]\ff (x) := \{y\in X; \ff^n(y) = \ff^n(x)\}.$$
The cardinality of $\l \ff(x)$ (resp. $\l[n]\ff(x)$)
will be denoted by $l^\ff(x)$ (resp. $l_n^\ff(x)$). 
 
For reasons that will be explained below,
the class of $\ff(x)$, namely $\l \ff(\ff(x))$ will be denoted by
$\l[+]\ff(x)$. If there exists a point $y\in X$ 
such that $\ff(y)\in \l \ff(x)$, $\l \ff (y)$ is
denoted by $\l[-]\ff(x)$.
One can verify that $\l[+] \ff(x)$ is well-defined and
that both $\l[-]\ff(x)$ and $\l[+] \ff(x)$ are {\em class} objects; 
i.e., they do not depend on the choice of the element
of the equivalence class.

\begin{rem}\label{remhom}
For a homeomorphism $\ff$ on a metric space,
the {\em stable manifold} of a point $x\in  X$ with respect to $\ff$ is
\eqn{W^s(\ff,x) = \{y\in X; \lim_{n\to\infty}d(\ff^n(x),\ff^n(y))=0\}.}
Hence, in the case where the space $X$ is equipped with a discrete metric,
the stable manifold foliation is the $\ff$-foliation of $X$ as defined above.
This explains the chosen terminology.
\end{rem}

The measurability of $\ff$ implies all foils are measurable subsets of $X$.

A partition $\L{}$ of  $X$ into measurable sets is called $\ff$-invariant
if for all $L\in\L{}$
\eqn{\ff^{-1}(L) = \{x\in X; \ff(x)\in L\}\in\L{},}
provided that $\ff^{-1}(L)\neq\emptyset$.

The $\ff$-foliation of $X$
is the finest $\ff$-invariant partition 
$\L[X]\ff$ of $X$
in the sense that for all $\ff$-invariant partitions $\L \prime$ one has 
\eqn{\forall L\in\L[X]\ff,\ \exists L^\prime\in\L \prime \mbox{ s.t. } L\subset L^\prime. }

\begin{defi}
\label{def-graph}
The graph $G^\ff = G^\ff(X)=(V,E)$ has for set of vertices $V=X$
and for set of edges $E=\{(x,\ff(x)), x\in X$\footnote{In all cases
to be considered, the connected components of $G^\ff$
will always have a countable collection of nodes
and a finite degree, even when $X$ is not countable; see the next remark.}.
Note that this graph can be
considered either as undirected 
or as directed, with each edge from $x$ to $\ff(x)$.   

For $x\in X$, denote by  $C^\ff(x)$ the undirected
connected component of $G^\ff$ which contains $x$;
i.e., the set of all points $y\in X$ for which there exist
non-negative integers $m$ and $n$ such that $\ff^m(x) = \ff^n(y)$. 
The set of connected component of $G^\ff$ will be denoted
by $\C \ff (X)$.
If $x\sif y$ then $x$ and $y$ are in the same connected
component of $C^\ff(x)$. In other words, the foliation is
a subdivision of $\C\ff (X)$.

$C^\ff(x)$ will be said to be \emph{$\ff$-acyclic}, if the restriction
of $G^\ff$ to $C^\ff(x)$ is a tree.
\end{defi}

\begin{lem}\label{cyclic.component}
The connected component $C=C^\ff(x)$ of $G^\ff$ is either
an infinite tree or it has exactly one (directed) cycle $K(C)$; in the  
latter case, for all $y\in C$, there exists $n\in \mathbb N$
such that $\ff^n(y)\in K(C)$. 
\end{lem}

\proof All statements follow from the fact that all
vertices of $C$, seen as a directed graph, have out-degree
equal to one and from the fact that $C$ is connected (as an undirected graph).

\begin{rem}\label{countable.comp}
If for all $x\in X$, $\mbox{Card}(\ff^{-1}(x))$ is finite,
then $C^\ff(x)$ is countable. 
\end{rem}

Whenever it is clear from the context, the superscript $\ff$ is dropped.

\subsection{Foil Order}
\label{sec:ford}
The $\ff$-foliation of each connected
component of $X$ can be equipped with some form of order. Consider $\ff(x)$
as the {\em father} of $x$.  Then $\l \ff(x)$ denotes the $\ff$-generation
of $x$ i.e., the set of its $\ff$-cousins of all orders;
$\l[n]\ff(x)$ denotes the set
of its $\ff$-cousins with common $n$-th $\ff$-ancestor. 
In addition, $\l[+]\ff(x)$ is the $\ff$-generation {\em senior} 
to $x$'s, i.e., that of its father, whereas
$\l[-]\ff(x)$  (if it exists) is the $\ff$-generation {\em junior} to $x$'s, 
i.e., that of its sons (if any) or that of the sons of its cousins (again if
any).

\begin{defi}\label{foils.order}
Note that if $C(x)$ is acyclic, this definition of generations
gives a linear order on the foils of $C(x)$ which is that of seniority:
by definition $\l \ff(y)<\l[+]\ff(y)$ for all $y\in C(x)$.
This order is then similar to the order of either $\Bbb Z$ or $\Bbb N$
(total order with either no minimal element or with a minimal element).

\end{defi}

Note that $\ff^n(X)$ is a sequence of decreasing sets in $n$.
Its limit (which may be the empty set) is denoted by $\ff^\infty(X)$
and, consistently with the seniority order, the set $\ff^\infty(X)$
will be called the set of \emph{$\ff$-primeval} elements  of $X$.   

\begin{defi}\label{defmn}
Let $n$ be a positive integer.
For all $x\in X$, let $D_n(x)=D_n(\ff,x)$
be the set of all descendants of $x$ which belong to
the $n$-th generation w.r.t. $x$; i.e.,
\eqn{D_n(x):=\{y\in X; \ff^n(y)=x\}.}
The cardinality of $D_n(x)$ 
(which may be zero, finite or infinite)
is denoted by $d_n(x)$. Also, let $D(x)=D(\ff,x)$ denote the set
of all descendants of $x$; i.e.,
\eqn{D(x):=\{y\in X; \exists~ n\ge 0: \ff^n(y)=x\}=\bigcup_{n=1}^\infty D_n(x).}
Finally the cardinality of $D(x)$ is denoted by $d(x)$. 
\end{defi}

\section{Point Processes and Point-Shifts}
\label{sec:ppps}
Whenever $(\sp,+)$ acts on a space,
the action of $t\in\sp$ on that space is denoted by $\theta_t$.
It is assumed that $(\sp,+)$ acts on the reference
probability space $(\Omega,\sf)$.
\subsection{Counting Measures and Point Processes}
\label{sec:not}
Let $\n$ be the space of all locally finite and simple counting measures
on $\sp$. It contains all measures
$\phi$ on $\sp$ such that for all bounded (relatively compact)
Borel subsets $B$
of $\sp$, $\phi(B)\in \mathbb N$ (counting measure condition)
and for all $x\in \sp$, $\phi(\{x\})\le 1$ (simplicity condition).
Let $\N$ be the cylindrical
$\sigma$-field on $\n$ generated by the
functionals $\phi\mapsto\phi(B)$, where $B$ ranges over
the elements of $\B$, the Borel $\sigma$-field of $\sp$.
The flow $\theta_t$ acts on counting measures as
$$ (\theta_t \phi)(B) = \phi(B+t),$$
and therefore on $\sp$ as $ \theta_t x = x-t$.

Let $\n^0$ be the subspace of $\n$ of counting measures
with an atom at the origin.

A $\emph{(random) point process}$ is  a
couple $\php$ where $\p$ is a probability measure on a
measurable space $(\Om,\sf)$ and $\Phi$ is a measurable mapping
from $(\Om,\sf)$ to $(\n,\N)$. Note that the point process $\php$ is
a.s. simple by construction.

The stationarity of a point process translates into the assumptions
that for all $t\in \sp$, $\theta_t \p=\p$ and
and that $\Phi(\theta_t \omega)= \theta_t \Phi(\omega)$.

When the point process $\php$ has a finite and positive intensity,
its Palm probability \cite{DaVe08} is denoted by $\p_\Phi$.
Expectation w.r.t. $\p_\Phi$
is denoted by $\fe$.

\subsection{Flow-Adapted \PoM s} 
A \pom{} on $\n$ is a measurable function  $\f$ which is
defined for all pairs $(\phi, x)$, where $\phi\in\n$
and $x\in\phi$, and satisfies the relation
$\f_\phi(x) \in \phi.$

In order to define flow-adapted \pom{}s,
it is  convenient to use the notion of \pomk.
A measurable function $\g$ from the set $\n^0$ to $\sp$
is called a \emph{\pomk} if for all 
$\phi$ in $\n^0$, $\g(\phi)$
belongs to $\phi$.
 
If $\g$ is a  \pomk, the associated
\emph{flow-adapted \pom}, $\f=\f_\g$,
is a function which is defined for all pairs $(\phi,x)$,
where $\phi\in\n$ and  $x\in\phi$, by 
$\f_\phi(x) = \g(\theta_x\phi) + x.$
The \pom{} $\f$ is flow-adapted because 
\begin{eqnarray}\label{compf}
\f_{\theta_t\phi}(\theta_tx) &=& \f_{\theta_t\phi}( x -t)=
\g(\theta_{x -t}(\theta_t\phi))+x-t\nonumber\\
&=&
\g(\theta_x\phi)+x-t=\f_\phi(x)-t=\theta_t(\f_\phi(x)).
\end{eqnarray}
In the rest of this article, \pom{} always means 
flow-adapted \pom{}. \Pom{}s will be denoted by
capital letters and the \pomk{} of a given \pom{}
will be denoted by the associated 
small letter ($\f$'s \pomk{} is hence denoted by $\g$).

The \emph{$n$-th image of
$\phi$ under $\f$}
is inductively defined as 
\eqn{\ffsubph^n\phi=\ffsubph(\ffsubph^{n-1}\phi), \quad n\ge 1,} 
with the convention $\ffsubph^{0}\phi=\phi$.
Notice that $\ffsubph^n\phi$ is not necessarily simple.

\subsection{Examples}
\label{examples}
This subsection introduces a few basic examples which will be 
used to illustrate the results below. 
These examples will be based on two types of point processes:
Poisson point processes and 
Bernoulli grids. The latter are defined as follows: it is well known that
the $d$ dimensional lattice $\Z^d$ can be transformed into a
stationary point process in $\R^d$ by a uniform
random shift of the origin in the $d$ unit cube.
The \emph{Bernoulli grid}
of $\R^d$ is obtained in the same way when
keeping or discarding each of the lattice points 
independently with probability $p$.
The result is again a stationary point process whose distribution will 
be denoted by ${\cal P}_p$.
\subsubsection{Strip  \PoM{}}
\label{ssStrip}
The Strip \PoM{} was introduced by Ferrari,
Landim and Thorisson \cite{FeLaTh04}. 
For all points $x=(x_1,x_2)$ in the plane, let $St(x)$ denote the 
half strip $(x_1,\infty)\times [x_2-1/2,x_2+1/2]$. Then $S(x)$ is the
left most point of $St(x)$. It is easy to verify that $S$ is
flow-adapted. Denote its \pomk{} by $s$.

The strip \pom{} is not well-defined when there are more than one
left most point in $St(x)$, nor when the
point process has no point (other than $x$) in $St(x)$.
Note that such ambiguities
can always be removed, and some refined version
of the strip \pom{} can always be defined
by fixing, in some flow-adapted manner, 
the choice of the image and by choosing $f_\phi(x)=x$
in the case of non-existence. By doing so one gets a refined
\pom{} defined for all $(\phi,x)$.

\subsubsection{Mutual Nearest Neighbor \PoM{}}
\label{ssMnn}
Two points $x$ and $y$ in $\phi$ are mutual nearest neighbors 
if $\phi(B^o(x,||x-y||))=\phi(B^o(y,||x-y||))=1$ and
$\phi(B(x,||x-y||))=\phi(B(y,||x-y||))=2$,
where $B^o(z,r)$ (resp. $B(z,r)$) denotes the open
(resp. closed) ball of center $z$ and radius $r$.
The Mutual Nearest Neighbor \PoM{} $N$
is the involution which
maps $x$ to $y$ when these two points are mutual nearest neighbors
and maps $z$ to itself if $z$ has no mutual nearest neighbor.
This \pom{} is bijective.

\subsubsection{Next Row \PoM{} on the Bernoulli Grid }
\label{ssGrid}

The \emph{Next Row} \pom{},
which will be denoted by $R$, is defined on the $d$-dimensional 
Bernoulli grid as follows:
\eqn{R_\phi(x_1,\ldots, x_d) = (x_1+1, x^\prime_2, x_3,\ldots, x_d),}
where 
\eqn{x^\prime_2 = \min \{y\geq x_2; (x_1+1,y, x_3,\ldots, x_d)\in\phi\}).}
It is easy to verify that if $d\geq 2$ and $p>0$, $R$ is a.s. well-defined. 

\subsubsection{Condenser \PoM{}s}
\label{ssCEPM}
Assume each point $x\in\phi$ is marked with 
\eqn{m_c(x)=\#(\phi\cap B(x,1)).}
Note that $m_c(x)$ 
is always positive. The 
\emph{condenser \pom} 
acts on  marked 
point process as follows: it goes from each point $x\in\phi$
to the closest point $y$ with a larger first coordinate such that
$m_c(y)= m_c(x) + 1$.
It is easy to verify that the condenser 
\pom{} is flow-adapted and almost surely well-defined
on the homogeneous Poisson point process.

\subsection{On Finite Subsets of Point Process Supports}
This subsection contains some of the key
technical results to be used in the proofs below.
Below, a counting measure will often be
identified with its support, namely with a discrete subset of
$\sp$.
\begin{lem}\label{counting.given.point}
Let ${\cal X}\subset \n(\sp)$ be a family of discrete subsets of
$\sp$ such that, for all $t\in\sp$ and $X\in\cal X$,
one has $\theta_tX\in\cal X$. Assume $\pi:{\cal X}\to 
\n(\sp)$ is a measurable finite and non-empty
\emph{flow-adapted inclusion}, i.e., for all $t\in\sp$ and all $X\in\cal X$,
\begin{eqnarray*}
0<\norm{\pi(X)}<\infty && \text{(finite and non-empty),} \\
\pi(\theta_t X)=\theta_t\pi(X) && \text{(flow-adapted),}\\
\pi(X)\subset X && \text{(inclusion).}
\end{eqnarray*}
Then, there exists a flow-adapted numbering of the points of the elements
of $\cal X$.  
\end{lem}

\begin{proof}
For all $X\in\cal X$ one can choose a point $y_1$ of
$\pi(X)$ in a flow-adapted manner;
e.g. the least point in lexicographic order of $\R^d$.
Then considering $y_1$ as the first point,
one can number the other points of $X$ according to
their distance to $y_1$ in an increasing order;
if there are several points equidistant to $y$,
one can sort them in increasing lexicographic order.
Note that the fact that $X$ is discrete implies
there are at most finitely many points in $B_r(y)$
for all given $y$ and hence the above numbering is well-defined. 
\end{proof}

\begin{thm}\label{no.finite.invariant.collection}
Let $(\Phi,\p)$ be a stationary point process and $n=n(\Phi)$ be
a measurable flow-adapted random variable taking
its values in $\overline{\Bbb N}= \Bbb N\cup\{\infty\}$. 
Let $\Psi=\{\Psi_i\}_{i=1}^n$ be a flow-adapted
collection of infinite, pairwise disjoint measurable
subsets of $\Phi$. If $\Xi=\{\Xi_i\}_{i=1}^n$ is a
flow-adapted collection of subsets of $\Phi$ such
that for all $i$, $\Xi_i$ is a finite subset of $\Psi_i$ then,
\asy, all $\Xi_i$-s are empty.   
\end{thm}

In words, Theorem~\ref{no.finite.invariant.collection} states that no
stationary point process (and no collection of its infinite disjoint subsets)
possesses a finite non-empty flow-adapted inclusion.

\begin{proof}
If, for some $i$, $\Xi_i\neq\emptyset$, one has a
flow-adapted numbering of the points of $\Psi_i$
from Lemma~\ref{counting.given.point}.
Let $A(m,\Xi_i)$, $m\ge 1$, be the $m$-th point of $\Psi_i$ in this numbering.
Now define the following point-shift:
\eqn{
\f_\Phi(x) = 
\begin{cases}
x & x\in\Phi\backslash\bigcup_{i\ \mbox{s.t.}\ \Xi_i\neq\emptyset} \Psi_i,\\
A(m+1,\Xi_i) & \Xi_i\neq\emptyset, x = A(m,\Xi_i).
 \end{cases}}
The compatibility assumptions imply that $\f_\Phi$ is indeed a \pom{}.
It is clear from the definition that $\f_\Phi$ is
injective from the support of $\Phi$ to itself, and according
to Corollary~\ref{surjective.injective} below, $\f_\Phi$ is \asy{} bijective.
But, for all $i$ for which $\Xi_i\neq\emptyset$, 
$A(1,\Xi_i)$ is not the image of any point.
Therefore \asy{} there is no non-empty $\Xi_i$. 
\end{proof}

\begin{rem}
In Theorem~\ref{no.finite.invariant.collection},
the condition that the $\Xi_i$-s are disjoint 
is not necessary. But for the sake of space,
the proof of this more general case is skipped. 
\end{rem}

\begin{cor}\label{no.finite.invariant.subset}
Letting $\Psi=\Phi$ in Theorem~\ref{no.finite.invariant.collection}
gives that one cannot choose a finite non-empty subset
of a point process in a flow-adapted manner. 
\end{cor}

\subsection{Partitions of the Support of a Point Process}
\label{sec:parti}

A {\em partition of counting measures},
${\cal T}$, is a map that associates to each $\phi\in\n$ 
a partition $\ct(\phi) = \{T_n(\phi); n\in \Bbb N\}$ 
of the support of $\phi$ into a countable collection of non-empty sets.

This partition is flow-adapted if for all $\phi\in\n$ and all $t\in\sp$,
\eqn{{\ct} (\phi) =
\{ T_n(\phi),\ n\in \Bbb  N\} \Rightarrow
\ct(\theta_t\phi)= \{ \theta_t T_n(\phi); n\in \Bbb N\}.}
One of the simplest cases of flow-adapted partitions
is the {\em singleton partition}; i.e., 
\eqn{\ct(\phi)=\{\{t\}; t\in\phi\}.}
For a partition $\ct$,
the element of $\ct(\phi)$ that contains $t\in\phi$ is denoted 
by $\ct_t(\phi)$. Using this notation, it is easy to see that 
each flow-adapted partition ${\cal T}$ of counting measures
is fully characterized by a measurable map $\ct_0:\n^0\to\n^0$.  Indeed, 
\begin{equation} {\cal T}_t(\phi)= \theta_{-t}\ct_0(\theta_t \phi) = t+\ct_0(\theta_t \phi).
\end{equation}

An enumeration of the elements of a set is an 
injective function $\nu$ from from this set to $\Bbb N$
(or equivalently to  $\Z$).
There are several enumerations of the elements of the partition ${\cal T}$;
e.g. based on the distance to the origin. Any element $T$ of the partition
is a countable collection of points of $\phi$. Since $\phi$
has no accumulation points, one can define the distance 
of $T$ to the origin as the minimum of the distances from
the points of $T$ to the origin. If the set of distances to 
the sets of the partition are all different, one defines
$T_0$ as the element of the partition with the smallest distance
to the origin, $T_1$ as the one with the
second smallest distance to the origin, and so on.
Ties are treated in the usual way, e.g. by using lexicographic ordering.
Note that this enumeration is not flow-adapted.

A natural question
is about the existence of translation
invariant enumerations.
This is not always granted.
For example, it is well known, and can be seen from
Corollary~\ref{no.finite.invariant.subset}, that
the singleton partition of a stationary point process $(\Phi,\p)$ 
cannot be enumerated in a measurable and flow-adapted manner.

\begin{defi}
\label{de:marka}
A flow-adapted partition of a
stationary point process
will be said \emph{markable}
if there exists an enumeration of the elements of the
partition which is invariant by translations.
\end{defi}

\begin{prop}
The flow-adapted partition $\ct$ of the stationary point process $(\Phi,\p)$ is 
markable if and only if $(\Phi,\p)$ can be partitioned as
\eq{\label{markable.decomposition}\Phi=\bigcup_{i=1}^n \Phi_i,\quad n\in\overline{\Bbb N},} 
where $(\Phi_i,\p)$ is a sub-stationary point process such that  
its support is an element of $\ct$.
\end{prop}

\begin{proof}
Assume $\ct$ is markable with the enumeration function $\nu$. 
Then, 
\eqn{\Phi=\bigcup_{i=1}^n \Phi_i:=\bigcup_{i=1}^n\nu^{-1}(i).}
The injectivity of $\nu$ implies $\nu^{-1}(i)$ is either
an element of $\ct$ or the empty set with positive probability.
In addition, the translation 
invariance of $\nu$ gives $\nu^{-1}(i)$ is a sub-stationary point process. 

On the other hand if $\Phi$ possesses a \emph{markable
decomposition}~(\ref{markable.decomposition}), then
the function $\nu$ which assigns to each
element $T$ of $\ct$, the index $i$ for which the support
of $\Phi_i$ is $T$, is an enumeration function of for $\ct$ 
and hence $\ct$ is markable. 
\end{proof}

The reason for this terminology is that if the partition is defined
by a selection of the points of $\Phi$ based on marks (see e.g.
\cite{DaVe08} for the definition of marks of a point process),
then such an enumeration exists\footnote{In fact the following
result holds: there exists an enumeration invariant by translation
if and only if there exists a decomposition of the stationary point process
into a collection of stationary sub-point processes with  
disjoint supports and with positive intensities. The proof of this result
is skipped as it will not be used below.}.
For instance, the singleton partition of a stationary point process is
flow-adapted but is not a markable partition.  

\begin{defi}\label{partition.preserving.ps}
Let ${\cal T}$ be a flow-adapted partition of the support of $\Phi$.
Let $\h$ be a  \pom{}. One says $\h$ \emph{preserves} ${\cal T}$
if for all $T\in{\cal T}$, $\h^{-1}(T)=T$.
If $\h$ is bijective, this is equivalent to the property
that for all $T\in{\cal T}$, $\h(T)=T$.
\end{defi}
\begin{defi}\label{stable.group}
Let $\Gamma_{\cal T}:=\Gamma_{\cal T}(\Phi)$ be the set of all bijective
and $\cal T$-preserving \pom s. The set $\Gamma_{\cal T}$ can be equipped
with a group structure by composition of \pom s.
This group, which is as subgroup
of the symmetric group on the support
of $\Phi$, is called the \emph{$\cal T$-stable group}.
An element $\h$ of this $\cal T$-stable group is said \emph{$\cal T$-dense}
if $\p$-\asy, for all $x\in\phi$, the orbit of $x$ under $\h$
spans the whole set of the partition that contains $x$; i.e.,
\eqn{\{H^n(x);n\in\Bbb Z\}=\ct_x(\phi).}
\end{defi}

\section{\PoMk{} Foliations}
\label{sec:psfo}
This section introduces two dynamics associated with
a flow-adapted \pom{} $\f=\f_\phi$ (or equivalently to its associated \pomk{} $\g$)
and discuss the associated foliations.
\begin{enumerate}
\item \label{ds1} For all fixed $\phi\in \n$, consider
the map $g=\f_\phi$, 
from the discrete space $\mathrm {support} (\phi)$ to itself.
The $\f_{\phi}$-foliation of $\phi$ is a partition of the set
$\mathrm {support} (\phi)$. It will be denoted by $\L [\phi]{\f_{\phi}}$.
The set of connected components will be denoted by
$\C [\phi]{\f_{\phi}}$. 
Whenever the context allows it, the subscript $\phi$ is dropped, so that
the latter set is denoted by $\C \f$
and the former by $\L \f$.
\item \label{ds2} $(\n^0,\thetg)$: for all $\phi\in\n^0$,
let $g(\phi)=\thetg\phi:=\theta_{\g(\phi)}\phi$. The map $\thetg$
is a measurable dynamics on $\n^0$, a non-discrete measure space.
The definition of the $\thetg$-foliation is nevertheless that
of Definition \ref{foliation}\footnote{ Rather than that of the stable
manifold alluded to in Remark \ref{remhom}}. The reason for this
choice of definition is given in Corollary \ref{folfg} below.
The associated foliation (resp. set of
connected components) is denoted by
$\L[\n^0]\thetg$ or simply by $\Ltg$ (resp.
$\C[\n^0]\thetg$ or $\Ctg$). 
Note that this partition of $\n^0$ is very different
in nature from that discussed for dynamics 1 above:
each connected component of $\Ltg$
(and hence each foil or each component of the graph $G^{\thetg}$)
is still discrete, whereas $\n^0$ is a non-countable set. 
So the number of connected components of this foliation must be non-countable.
\end{enumerate}
Although $\L \f$ and $\L\thetg$ are defined  on different spaces, they
are closely related because of the following statement, which follows from   
the compatibility of the point-shift $\f$.
\begin{cor}\label{folfg}
\eq{\label{fgfol}x\sim_{\f_{\phi}} y \Leftrightarrow \theta_x\phi\sim_{\thetg}\theta_y\phi.} 
\end{cor}

\begin{exm}
Consider the Next Row \pom{}, R, on $d$-dimensional Bernoulli grid defined in
Subsection \ref{ssGrid}. If $p\in(0,1)$, one can show that 
\eqn{L^R(x_1,\ldots,x_d) = \{(x_1,y,x_3\ldots,x_d)\in \Phi\},}
and 
\eqn{C^R(x_1,\ldots,x_d) = \{(y_1,y_2,x_3\ldots,x_d)\in \Phi\}.}
Thus each foil looks like a 1-dimensional Bernoulli grid and each
connected component looks like a 2-dimensional Bernoulli grid (Figure~\ref{next.row.foils}). If $d=2$,
the graph $G^R$ has a singe connected component.  
\begin{figure}[h]
\centering
\includegraphics[width=0.5\linewidth]{{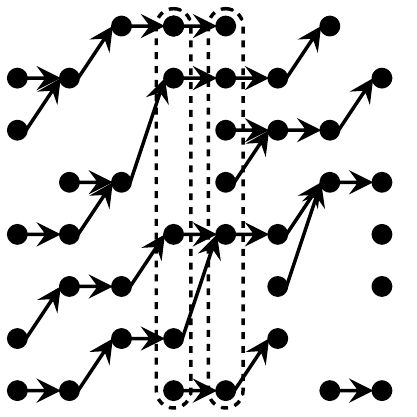}}
\caption{The Next Row \pom{} on 2-dimensional Bernoulli grid. Dashed lines indicates 
two foils of this \pom.}
\label{next.row.foils}
\end{figure}

\end{exm}

Consider now a stationary point process $(\Phi,\p)$,
with Palm version denoted by $(\Phi,\fp)$.
The expectation with respect to $\p$ (resp. $\fp)$
is denoted by $\e$ (resp. $\fe$). 
The above dynamics lead to the following stochastic objects:
\begin{enumerate}
\item $\f_\Phi$ is a random map from the discrete random set
$\mathrm {support}(\Phi)$ to itself.
Both the component partition $\C\f =\C [\Phi]{\f_{\Phi}}$ and
the foil partition $\L\f = \L [\Phi]{\f_{\Phi}}$ are flow-adapted
partitions of this random set,
with the latter being a refinement of the former.
\item $\L\thetg=\L[\n^0]\thetg$ is a deterministic
partition of the whole set $\n^0$
(in contrast to the random partition of a random set described above).
Note however that it is sufficient that $\thetg$ be defined
$\fp$-almost surely and hence it may be undefined
for some elements of $\n^0$ of null measure for $\fp$.
\end{enumerate}
Here are some observations 
on the flow-adapted partitions
$\C{\f}$
and $\L {\f}$.
These two partitions do not depend on $\p$ at all (since
they are defined on realizations). In particular, they do not
depend on whether the point process is considered under $\p$ or $\fp$. 

The elements of each of these two partitions can be enumerated in a natural
way following the method discussed just before Definition \ref{de:marka}.
 
The dichotomy of Section \ref{sec:parti} applies:
there are cases where $\L {\f}$ (resp. $\C {\f}$)
is a markable partition
and cases where it is not\footnote{The partition $\L [\Phi]{\f_{\Phi}}$
gathers points whose marks are in the
same equivalence class w.r.t. some equivalence relation.  
This does not mean that this partition is markable.}.
A simple instance of the latter case
is obtained when $\f$ is bijective;
then the foil partition coincides with the singleton partition,
which is not a markable partition.

The following solidarity properties hold:
\begin{prop}
\label{thdic}
If the foil partition $\L{\f}$ is markable, 
so is the component partition $\C{\f}$.
Conversely, if $C$ is a component which is the support
of a flow-adapted point process $\Xi$, then
either the foil partition of $\Xi$, $\L [\Xi]{\f_{\Xi}}$
is markable or there is no flow-adapted point process
with a positive intensity
having for support a foil of $\L [\Xi]{\f_{\Xi}}$.
\end{prop}
\begin{proof}
The first assertion is immediate.
The proof of the converse leverages the foil order 
introduced in Subsection \ref{sec:ford}.
First observe that $\Xi$ has a single component.
Assume that, for some foil $L$ of $\Xi$,
the point process $\Psi(L)$ with support $L$
is a flow-adapted point process.
Then $\Psi(L_+) =\Psi(\f(L))$
is a flow-adapted point process with non empty support (as $L$ is non empty)
and hence with positive intensity.
Hence all foils that are senior to $L$ are flow-adapted point processes
with a positive intensity.
Similarly, either $L_-$ is empty, and then there is no foil
junior to $L$, 
or $\Psi(L_-) =\Psi(\f^{-1}(L))$ is also
a flow-adapted point process with a positive intensity. 
It then follows that the foil partition of $\Xi$ is markable.
\end{proof}

\begin{rem}
\label{rem:enumf}
Under $\fp$, the foil order leads to a natural
enumeration of the foils of the component of the origin.
The foil of the origin is numbered 0 and will be
denoted by $L^\f(0)=L_0$, the foil senior (resp. junior)
to it will be numbered 1
and will be denoted by $L_1$ (resp. $L_{-1}$
if non empty), and so on.
Note that this enumeration is not flow-adapted.
\end{rem}

\section{\PoMk{} Cardinality Classification}
\label{sec:cacl}
In the rest of this work $(\Phi,\p)$ is a stationary point process 
with Palm version $(\Phi,\fp)$.

The foliation $\L {\f}$ partitions the
support of the point process $\Phi$
into a discrete set of connected components;
each component is in turn decomposed in a discrete
set of $\f$-foils, and each foil
in a set of points. The present subsection proposes a classification 
of point-maps based on the cardinality of these sets.

\subsection{Connected Components}
The cardinality classification
of connected components of the two dynamics differ.

The partition $\C[\Phi] {\f_{\Phi}}$
is countable. Its cardinality is a random variable
with support on the positive integers and possibly infinite.
If $(\Phi,\p)$ is ergodic, this is a positive constant or $\infty$ \asy.

As already mentioned, in contrast, the partition
$\C\thetg$ is deterministic and non-countable in general.

\subsection{Inside Connected Components}
In view of Corollary \ref{folfg}, the
cardinality classification 
of the foils belonging to a given connected component
is the same for $\L [\Phi]{\f_{\Phi}}$ and for $\L\thetg$.

It is easy to see that the cardinality of the set of foils of a component
is a random variable with support in $\overline{\Bbb N}= \Bbb N\cup\{\infty\}$. 
The same holds true for the set
of points of a non-empty foil. The following theorem shows that
only a few combinations are however possible:
\begin{thm}[Cardinality classification of a connected component]\label{cardinality.components}
Let $(\Phi,\p)$ be a stationary point process. Then $\p$
almost surely,  each connected component $C$ of $G^\f(\Phi)$
is in one of the three following classes:
\begin{description}
\item[Class~${\mathcal F}/{\mathcal F}$:] \label{finite} $C$ is finite, and hence so is each
of its $\f$-foils. In this case,
when denoting by $1\le n=n(C)<\infty$ the number of its foils: 
\begin{itemize}
\item $C$ has a unique cycle of length $n$;
\item $\f^\infty(C)$ is the set of vertices of this cycle.
\end{itemize}
\item[Class~${\mathcal I}/{\mathcal F}$:]\label{infinite.finite} $C$ is infinite
and each of its $\f$-foils is finite. In this case:
\begin{itemize}
\item $C$ is acyclic;
\item Each foil has a junior foil, i.e., a predecessor for the order
of Definition~\ref{foils.order};
\item $\f^\infty(C)$ is a unique {\em bi-infinite path,}
i.e., a sequence of points $(x_n)_{n\in \Z}$ of $\phi$
such that $\f_\phi(x_n)=x_{n+1}$ for all $n$.
\end{itemize}
\item[Class~${\mathcal I}/{\mathcal I}$:]\label{infinite.infinite} $C$ is infinite and
all its $\f$-foils are infinite. In this case:
\begin{itemize}
\item $C$ is acyclic;
\item $\f^\infty(C)=\emptyset$.
\end{itemize}
\end{description}
\end{thm}

The following definitions will be used:

\begin{defi}\label{defev}
Let $C$ be a connected component of $G^\f(\Phi)$.
The point-shift \emph{evaporates}
$C$ if $\f_\Phi^\infty(C)=\emptyset$ \asy.
\end{defi}

It follows from Theorem \ref{cardinality.components} that:
\begin{cor}
\label{coreva}
The \pom{} $\f$ evaporates $C$ if and only if
$C$ is of Class~${\mathcal I}/{\mathcal I}$.
\end{cor}
Before giving the proof of Theorem \ref{cardinality.components},
a collection of preliminary results
(Proposition \ref{mn1} to Corollary \ref{connected.so.tree}) is presented.

\begin{prop}\label{mn1}
Let  $(\Phi,\p)$ be a  stationary point process 
with Palm probability $\fp$. Let $d_n(0)$ and $d(0)$ be as in
Definition~\ref{defmn} for the $\f$-foliation of $\phi=\Phi(\om)$.
One has 
\eq{\label{ftr}\forall n\ge 0:\ep\lk d_n(0)\rk=1.}
In particular, for all $n$, $d_n(0)$ is $\fp$-almost surely finite.
If, in addition, $G^\f(\Phi)$ is $\fp$-almost surely acyclic, then
\eqn{\ep\lk d(0)\rk=\infty.} 
\end{prop}
\begin{proof}
The map
\eqn{w(\phi,x,y):=\ind\{\f^n_\phi(x)=y\}}
is a flow-adapted mass transport (see \cite{LaTh09}).  
The first statement is hence an immediate consequence of Proposition \ref{w}.
For the second part, when $G^\f$ is acyclic,
the $D_n$-s form a partition of $D$ and hence  
\eqn{\ep\lk d(0)\rk=\sum_{n=1}^\infty \ep\lk d_n(0)\rk=\infty.}
\end{proof}
\begin{rem}
Note that $G^\f(\Phi)$ is $\fp$-\asy{} acyclic if and only if  
$G^\f(\Phi)$ is $\p$-\asy{} acyclic.  
\end{rem}

\begin{cor} \label{monotone.cardinal.foils}
Proposition~\ref{mn1} implies that the degrees of all vertices
in $G^\f(\Phi)$ are a.s. finite. Hence \asy, if $l^\f(x)=\infty$,
then for all positive $n$, $l^\f(\f^n_\phi(x))=\infty$.  
\end{cor}

\begin{cor}\label{surjective.injective}
The point-shift $\f_\phi$ is \asy{} surjective
on the support of $\phi$ if and only if it is \asy{} injective. 
\end{cor}
\begin{proof}
If $\f_\phi$ is surjective (resp. injective), then \asy{} 
$d_1(0)\geq 1$ (resp. $d_1(0)\leq 1$). Since $\fe[d_1(0)]=1]$,
\asy{} $d_1(0)=1$, and hence the \pom{} is bijective. 
\end{proof}

\begin{prop}\label{infinite.acyclic}
The connected component $C$ of $G^\f(\Phi)$ is acyclic if and only if
it is infinite. Hence $G^\f(\Phi)$ is acyclic if and only if it has
no finite connected component.  
\end{prop}

\begin{proof}
According to Lemma~\ref{cyclic.component} each connected component of
$G^\f(\phi)$ has at most one cycle. If the latter is finite,
it possesses exactly one cycle. This proves the first statement. 

Let $n=n(\Phi)$ be the number of connected components of $G^\f(\Phi)$
which are infinite and possess a cycle. Let $\Psi=\{\Psi_i\}_{i=1}^n$ 
denote the collection of such components.
Note that $n$ may be infinite. According to
Lemma~\ref{cyclic.component}, each $\Psi_i$ has exactly one
cycle. This cycle is a finite non-empty
flow-adapted subset of $\Psi_i$,
which contradicts Theorem~\ref{no.finite.invariant.collection}.
Therefore \asy, there is no such component.
 
\end{proof}

The next corollary follows from Lemma \ref{cyclic.component}. 
\begin{cor}\label{connected.so.tree} 
If $G^\f(\Phi)$ is almost surely connected, it is almost surely a tree. 
\end{cor}
\vspace{.2cm}

\begin{proof}[Proof of Theorem~\ref{cardinality.components}] 
The result for connected components of Class~${\mathcal F}/{\mathcal F}$
is an immediate consequence of Lemma~\ref{cyclic.component}. 

Assume $C$ is an infinite component.
According to Proposition~\ref{infinite.acyclic}, $C$ is acyclic. 
Consider the collection of all connected components with
both finite and infinite foils. Denote this
collection by $\Psi=\{\Psi_i\}_{i=1}^n$, where $n$ may be infinity.
Corollary~\ref{monotone.cardinal.foils} implies that each $\Psi_i$
should have a largest finite foil, say $\Xi_i$,
where the order is that based on seniority (Definition~\ref{foils.order}).
Therefore, $\Psi=\{\Psi_i\}_{i=1}^n$ and $\Xi=\{\Xi_i\}_{i=1}^n$
satisfy the assumptions of Theorem~\ref{no.finite.invariant.collection}
and this is hence a contradiction.
So, \asy, there is no connected component with both 
finite and infinite foils, which proves that each
acyclic component is either of Class~${\mathcal I}/{\mathcal F}$ or ${\mathcal I}/{\mathcal I}$.   

Let $C$ be a connected component of Class~${\mathcal I}/{\mathcal F}$.
Almost surely, $C$ cannot 
have a smallest foil. Otherwise the latter
would again be a finite flow-adapted subset of
the infinite connected component $C$, which contradicts
Theorem~\ref{no.finite.invariant.collection}.
This proves the second assertion on the foils of $C$ in this case.
Now let $L_0$ be an arbitrary foil of $C$ and,
for all integers $i$,
let $L_i$ be the foil $\f_\phi^i(L_0)$. 
Since $L_0$ is finite, there exists a least
non-negative integer $n$
such that $\f^n(L_0)$ is a single point.  Let
\eqn{C_0:=\{\f_\phi^m(L_0),\ -\infty<m<n\},}
The graph $G^\f(C_0)$ is infinite, connected 
and all its vertices are of finite degree.
It hence follows from K\"onig's infinity lemma \cite{Ko90}
that $G^\f(C_0)$ has an infinite path $\{x_i\}_{i\leq 0}$.
For $i>0$, define $x_i:=\f_\phi^i(x_0)$. 
Then $(x_i)_{i\in\Z}$ is a bi-infinite path.
Clearly $(x_i)_{i\in\Z}\subset \f_\phi^\infty(C)$.
Since all edges of $G^\f(C)$ are from a foil $L$ to the foil $L_+$,
$(x_i)_{i\in\Z}$ has exactly one vertex in each foil.
Finally for each point $y$ in an arbitrary foil $L(x_i)$,
there exists $m>0$ such that $\f_\phi^m(y)=\f_\phi^m(x_i)=x_{i+m}$
and hence $\f^\infty(C)\subset (x_i)_{i\in\Z}$,
which completes the proof of the properties of Class~${\mathcal I}/{\mathcal F}$.
 
Consider now $C$ of Class~${\mathcal I}/{\mathcal I}$ and assume that $\fph^\infty(C)$ is not empty.
If $x$ is a primeval element of $C$, then $G^\f(D(x))$ is an infinite connected
graph with vertices of finite degree and hence it possesses
an infinite path, which in turn gives a bi-infinite path using the
same construction as what was described above for the Class~${\mathcal I}/{\mathcal F}$.
Hence, in order to prove that
$\fph^\infty(C)$ is empty, it is sufficient to
show that $C$ has no bi-infinite path. 
If $C$ has finitely many bi-infinite paths, then the intersections
of these bi-infinite paths with each foil of $C$ give a collection
of finite subsets of infinite sets, which contradicts
Theorem~\ref{no.finite.invariant.collection}.
Consider now the case where $C$ has infinitely-many bi-infinite paths.
Since $C$ is connected, each two bi-infinite paths should intersect
at some point. Let $J=J(C)$ be the set of all points $x\in C$ 
such that at least two bi-infinite paths join at $x$.
It is now shown that, \asy,
the intersection of a bi-infinite path and $J$ has neither a first
nor a last point for the order induced by $\f$.
If it has a first (resp. last) point, then the part of the
path before the first (resp. after the last) point
is an infinite flow-adapted set with 
a finite flow-adapted subset, which contradicts
Theorem~\ref{no.finite.invariant.collection}.
Therefore, for each point $x\in J$, there is a smallest
positive integer $n(J,x)$ such that $\fph^{n(J,x)}\in J$.  
Now define a point-shift $h$ on the whole point process as follows
\eqn{
h(x)=
\begin{cases}
\fph^{n(J,x)} & x\in J(C)\\
x & \text{otherwise}.
\end{cases}
}
Since the intersection of any bi-infinite path with $J$
does not have a first point, $h$ is \asy{} surjective.
But from the very definition of $J$, all points of 
this set have at least two pre-images, which contradicts
Corollary~\ref{surjective.injective}. Hence the situation
with infinitely-many bi-infinite paths is not possible either, which
concludes the proof.
\end{proof}

In graph theoretic terms,
one can summarize the results discussed in the last
proof as follows:
\begin{cor}
\label{cor-1-2}
A Class~${\mathcal I}/{\mathcal I}$ component has one (positive) end.
A Class~${\mathcal I}/{\mathcal F}$ component has two ends (a positive and a negative one).    
\end{cor}

Theorem~\ref{cardinality.components} also has the following corollary:

\begin{cor}
For all stationary point processes $(\Phi,\p)$, for all
point-shifts $\f$, there exist three stationary point processes
$(\Phi_{{\mathcal F}/{\mathcal F}},\p)$, $(\Phi_{{\mathcal I}/{\mathcal F}},\p)$ and $(\Phi_{{\mathcal I}/{\mathcal I}},\p)$ 
(which may be empty with positive probability),
all defined on the same probability space, and such that
$$\Phi= \Phi_{{\mathcal F}/{\mathcal F}}+\Phi_{{\mathcal I}/{\mathcal F}}+\Phi_{{\mathcal I}/{\mathcal I}}.$$
All connected components of 
$G^\f(\Phi_i)$ are of Class~$i$, $i\in\{{\mathcal F}/{\mathcal F},{\mathcal I}/{\mathcal F},{\mathcal I}/{\mathcal I}\}$.
If $(\Phi,\p)$ is ergodic, then each of these point
processes is also ergodic.    
\end{cor}

\begin{proof}
The statement is an immediate consequence of
Theorem~\ref{cardinality.components} and the fact
that being a connected component of Class~$i$ is a flow-adapted
property. Thus if $\Phi_i$ is defined as the set of all points
in components of Class~$i$, $\Phi_i$ is flow-adapted.
Stationarity and ergodicity depend only on $\p$ and
the flow on the probability space and they are being
carried to new point processes. 
\end{proof}

\subsection{Comments and Examples}
\label{sec:exab}
Here are a few observations on the cardinality classification.

A point process $(\Phi,\p)$ can have a mix of components of all three
classes.

If $(\Phi,\p)$ only has ${\mathcal F}/{\mathcal F}$ components,
then it should have an infinite
number of connected components. An example of this situation
is provided by the Mutual Nearest Neighbor \PoM{} $N$ 
on Poisson point process on $\R^2$ (see below).

If $(\Phi,\p)$ only has ${\mathcal I}/{\mathcal F}$ components,
then the cardinality of $\C [\Phi]{\f_{\Phi}}$ may be finite
or infinite.
An example of the first situation is provided by
the Royal Line of Succession
\PoM{} on Poisson point processes on $\R^2$
(see below).
An example of the latter is provided by the Strip
\PoM{} $S$ on the Bernoulli grid of dimension 2.

If $(\Phi,\p)$ only has ${\mathcal I}/{\mathcal I}$ components,
then the cardinality of $\C [\Phi]{\f_{\Phi}}$ may again be finite
or infinite.
An example of the first situation is provided by
the Strip
\PoM{} $S$ on Poisson point processes on $\R^2$.
An example with infinite cardinality
is provided in Subsection \ref{ssMultiStrip} of the appendix.

The end of this section gathers 
examples of the three classes.

\subsubsection{Class~${\mathcal F}/{\mathcal F}$}
For the Mutual Nearest Neighbor \PoM{} $N$
on a Poisson point process,
there is no evaporation and there is
an infinite number of connected
components, all of class~${\mathcal F}/{\mathcal F}$.
The order of the foils is that of $\Z$ mod 2 or $\Z$ mod 1.
The foil partition is
the singleton partition, hence
not markable.
The connected component partition is not markable either.

\subsubsection{Class~${\mathcal I}/{\mathcal F}$}
An example of this class is provided by the strip \pom{}
on 2-dimensional Bernoulli grids.
It is easy to see that the connected  components of $\f_s$ 
are the {\em horizontal
sub-processes} which look like 1-dimensional Bernoulli grid. 
The \pom{} $\f_s$
is a bijection which leaves each of the connected components
invariant. It follows that there is a countable
collection of connected components. Each of them
is of Class~${\mathcal I}/{\mathcal F}$, and hence not markable. The foil
of a point is the singleton containing this point.
The set of descendants of a point $x$ consists of those points 
on the same horizontal line which are located on left side of $x$. 
Each connected component has two ends. Its foliation has
the same order as $\Z$ but is not markable.

Another example of this class is
leverages the Strip \pom{} on a stationary Poisson point process 
$(\Phi,\p)$ of $\R^2$ and $\R^3$. This  \pom{}
has a single connected component \cite{FeLaTh04}. The RLS ordering
(see the proof of Proposition~\ref{existence.f.orthogonal})
hence defines a total order on $(\Phi,\p)$, which
is equivalent to that of $\Z$. This allows one to
define the RLS \PoM\ $\f_{\mbox{\small{rls}}}$ which associates to
$x\in \Phi$ the unique point $y\in \Phi$ such that $x$
comes next to $y$ in this total order. This \pom{} is clearly bijective.
Hence the foil of $x$ is $\{x\}$. The unique connected component
of this \pom{} is thus
of class~${\mathcal I}/{\mathcal F}$.
The unique connected 
component has two ends. Its foliation has the same order as $\Z$.
It is not markable.

Note that these two examples are bijective \pom s. But there are cases 
of type $\cal I/I$ which are not bijective.  

\subsubsection{Class~${\mathcal I}/{\mathcal I}$}

Here are three examples illustrating that this class of \pom{}s
can have either markable or non markable foliations.

\begin{prop}
The Next Row \pom{} on the 2-dimensional Bernoulli grid 
with $0<p<1$, has a single connected component of type 
$\cal I/I$. Foils of this connected component are not markable.
\end{prop}

\begin{proof}
There is a single connected component and each foil consists
of all points of the point process on a vertical line
(Figure~\ref{next.row.foils}). 
Therefore all foils are infinite and the connected component is of type 
$\cal I/I$. If foils of this connected component were markable, 
the partition of the point process into horizontal lines would form
a collection of infinite disjoint subsets of the Bernoulli grid.
The intersection of these subsets 
with a some fixed foil is a finite non-empty inclusion, which contradicts
Theorem~\ref{no.finite.invariant.collection}.
\end{proof}

\begin{prop}
The Condenser \pom{} on the Poisson point process in $\R$ has 
a single $\cal I/I$ connected component and the foliation is markable. 
\end{prop}

\begin{proof}
It is easy to see that if $\lambda_n$ is the intensity 
of points with the mark $m_c$ equal to $n$, where 
$m_c$ is defined in Example~\ref{ssCEPM}, then $\lambda_n$ tends 
to zero as $n$ tends to $\infty$. Using this, one can conclude 
that the condenser \pom{} has a single connected component. 
Each foil consists of all points with the same mark $m_c$.
The foliation has the order of $\mathbb N$.
\end{proof}

Here are further examples of this class.
The authors in \cite{FeLaTh04} prove that  the Strip \pom{} $S$
on the Poisson point process in $\R^2$, 
has a single connected component which is one ended.
Therefore this connected component evaporates under the action of $\fst$ 
and is of Class~${\mathcal I}/{\mathcal I}$.
On a Poisson point process,
the expander point-shift also has a single connected component
which evaporate under the action of the \pom{}.
This component is of Class~${\mathcal I}/{\mathcal I}$.

\section{Partition Preserving Point-Maps}
\label{sec:pppm}

It is well-known that all bijective \pom s preserve
the Palm probability of stationary point processes
and that this property characterizes the Palm probabilities
of stationary point processes \cite{Me75}.

This section features
a fixed \pom{} $\f$ and considers the class 
of bijective \pom s which preserve the two partitions
$\C{\f} = \C [\Phi]{\f_{\Phi}}$ and
$\L \f = \L [\Phi]{\f_{\Phi}}$
of a stationary point process $\Phi$. 

\subsection {Bijective \PoM s Preserving Components}
Let $\Gamma_{\C{}}^\f:=\Gamma_{\C{\f}}(\Phi)$ be the 
$\C \f$-stable group as defined in Definition~\ref{stable.group}.
As mentioned in Definition~\ref{stable.group}, $\Gamma_{\C{}}^\f$
is a subgroup
of the symmetric group on the support
of $\Phi$.

\begin{prop}
\label{existence.f.orthogonalC}
For each \pom{} $\f$ and each stationary point process $(\Phi,\p)$,
there exists a $\C\f$-dense element (see Definition~\ref{stable.group}) of the
$\C\f$-stable group; i.e., there exists $H\in\Gamma_{\C{}}^\f$ such that 
for all $x\in \Phi$,
\eqn{\{\h^i(x);i\in\Z\}=\c[\Phi]\f(x).}
There is no uniqueness in general.
\end{prop}

\begin{proof}
The construction of $H$ is different for each of the three classes
of components identified in Theorem \ref{cardinality.components}.
In each case, the first step is the construction of a total order on the
points of $C$ which is flow-adapted and the second one
is the definition of a dense and bijective \pom{}
preserving $C$.

If $C=\c[\Phi]\f(x)$ is of ${\mathcal F}/{\mathcal F}$ class, then
it is easy to create a total order which is translation
invariant on the points of $C$ as
it is a finite set (e.g. using lexicographic order) with
points that can be numbered $0,1,\ldots,n-1$ for some integer
$n=n(x)\ge 1$. A flow-adapted bijection $H$
preserving $C$ is then easy to build by taking $H=M_n$
with $M_n(k)=k+1$ mod $n$.

If $C=\c[\Phi]\f(x)$ is of ${\mathcal I}/{\mathcal F}$ class, then,
the existence of a single bi-infinite path in $G^F(C)$ (see 
Theorem \ref{cardinality.components}) and the
finiteness of the foils can be used to construct a total order.
Let $\{x_n\}_{n\in \Z}$ be the bi-infinite path in question
and let
$L_n$ denote the foil of $x_n$.  
Since $L_n$ is finite for all $n$, one can use the lexicographic
order to create a total order between its points. The total order is
then obtained by saying that all points of $L_n$ have precedence over
those of $L_{n-1}$. This total order, which is that of $\Z$,
is flow-adapted.
The bijective \pom{} is that associating to a point $x$ of $C$ 
its direct successor for this order. This \pom{} will be referred
to as the Bi-Infinite Path \PoM{} $B$, with associated \pomk{} $b$.
On such a component, one takes $H=B$.

If $C=\c[\Phi]\f(x)$ is of ${\mathcal I}/{\mathcal I}$ class, then
the construction uses a total order on the nodes
of $G^\f(C)$ known as RLS (Royal Line of Succession).
The latter order is based on two ingredients:
\begin{enumerate}
\item A local (total) order
among the sons of a given node in $G^\f(C)$.
This can be done as follows: for a given point $x$ of $C$
let $B^\f(x)=B^\f_\phi(x)$ be the set of its brothers; i.e.,  
\eqn{B^\f(x):=\{y\in\phi;\f_\phi(x)=\f_\phi(y)\}. }
The elements of $B^\f(x)$ can then be ordered in a flow-adapted
manner using the lexicographic order of the Euclidean space.
\item The Depth First Search (DFS - see Appendix \ref{AB})
pre-order on rooted trees.
\end{enumerate}
The RLS order on a rooted tree is a total order
on a finite tree obtained by combining (1) and (2):
DFS is used throughout and the sons of any given node are visited
in the order prescribed by (1), with priority given to the
older son.

It is now explained how this also creates a total order on the nodes
of $C$. For $x,y\in C$, there exist positive integers $m$ and $n$ such
that $\fph^m(x)=\fph^n(y)$. One says that $x\geq_r y$ if $x$ has
RLS priority over $y$ in the rooted tree of descendants
of $\f_\phi^m(x)$. This tree is a.s. finite because  
in the ${\mathcal I}/{\mathcal I}$ case, there is evaporation
of $C$ by the \pom{}, which this in turn implies that
for all points $z\in \phi$, the total number of descendents
of $z$ is a.s. finite.
The DFS preordering on descendants of a node in a tree forms an
interval of this preordering. This implies that $\geq_r$ 
is a well-defined order on $C$ and also that it orders elements of
$C$ in the same linear order as that of an interval in $\Z$.
Furthermore, since $C$ is infinite,
this order on $C$ 
cannot have a greatest element or a least element.
Otherwise the greatest and the least elements would be a finite
flow-adapted subset of $C$ or the foil,
which contradicts Theorem~\ref{no.finite.invariant.collection}. 
Therefore the order on $C$ as well as its restriction to a foil
is a linear order, similar to that of $\Z$.

On such a components, one defined 
$H=R$ where $R$ denotes the RLS \pom{},
namely the \pom{} that associates to each point its 
successor in the RLS order, which is bijective and translation 
invariant.
\end{proof}
 
Let $h$ denote the \pomk{} of the \pom{} $H$ defined
in the last theorem.
Notice that since $H$ is bijective, the dynamical system $(\n^0,\theta_{h})$
preserves $\fp$.

\subsection{Bijective \PoM s Preserving Foils}
The results of this subsection parallel those of the
last subsection, with an important refinement which is
that of \emph{order preservation}.

Let $\Gamma_L^\f:=\Gamma_L^\f(\Phi)$ denote the set of all bijective
and $\L[\Phi]\f$-preserving \pom s.
This group, which is called the \emph{$\L \f$-stable group},
is a subgroup of the $\C \f$-stable group. 

As above, 
an element $\h$ of the $\L \f$-stable group
is said to be $\L[\Phi]\f$-dense if
\eqn{\{\h_\phi^i(x);i\in\Z\}=\l[\phi]\f(x).}

Each $\L[\Phi]\f$-dense element $\h$ of the $\L \f$-stable group
induces a total order $\preceq_\h$  on the elements of each infinite
foil of $C$ by 
\begin{equation}\label{eq:preceq}
x\preceq_\h \h(x).
\end{equation}
This total order is flow-adapted. It is said to be
preserved by $\f$ 
if \eqn{x\preceq_{\h}y\Rightarrow \f(x)\preceq_{\h}\f(y).}

The following proposition uses the fact
(proved in Theorem \ref{cardinality.components}) that in
a connected component $C$ of $G^\f(\Phi)$, either all foils of $C$
have finite cardinality or all foils have infinite cardinality.

\begin{prop}
\label{existence.f.orthogonal}
For each stationary point process $(\Phi,\p)$,
and each \pom{} $\f$, there exists a $\L[\Phi]\f$-dense element $\fo$ of the
$\L \f$-stable group. In addition $\fo$ can be chosen 
such that the $\preceq_{\fo}$ order is
preserved by $\f$
on components with all its foils with infinite cardinality.
There is no uniqueness in general.
\end{prop}

\begin{proof}
If the connected component $C$ of a realization $\phi$
has finite foils, the following construction can be used: $\fo(x)$ 
is the element coming next to $x$ in the lexicographic order.
This rule is applied to all elements
of a foil except the greatest element for this order, whereas
$\fo$ of the greatest element is the least element. 

For a connected component $C$ with all its foils with infinite cardinality, 
the construction uses the RLS total order on the nodes of $G^\f(C)$.

One defines $\fo(x)$ as the next element in $L(x)$, i.e.,
the greatest element of $L(x)$ which is less than $x$,
makes $\fo$ a bijection, and the orbit of each point $x$ of $L(x)$
is the foil $L(x)$. 

The property that $\preceq_{\fo}$ is preserved by $\f$ follows
from the fact that if $x$ has priority over $y$ for DFS, then
the father of $x$ also has priority
over the father of $y$ for DFS. 
\end{proof}

Let $\rfo$ denote the \pomk{} of the \pom{} $\fo$ defined
in the last theorem.
For the same reasons as above, the dynamical system $(\n^0,\theta_{\rfo})$
preserves $\fp$.

In the next definition and below,
in order to simplify notation, $\preceq_{\fo}$ 
(defined in (\ref{eq:preceq})) is often
replaced by $\preceq_\bot$.

\begin{defi}\label{f.orthogonal}
For two elements $x$ and $y$ of the same foil s.t. $x\preceq_\bot y$, let 
\eqn{\dd(x,y):=\mbox{Card}
\{z: y\preceq_\bot z\prec_{\bot} x\}.}
By convention let $\dd(y,x):=-\dd(x,y)$.
\end{defi}
It is easy to verify that for all $x$ and $y$ in the same foil,
\eq{\label{orthogonal.distance}\fo^{\dd(x,y)}(x)=y.}

\subsection {Point Foils and Components}
This subsection discusses some properties of the foil
and the component of the origin, seen as point processes.

For all countable sets $S$ of points of $\R^d$ without accumulation,
let $\Psi(S)$ denote the counting measure with support $S$.

Let $L_0$ (resp. $C_0$) denote the foil (resp. component)
of the origin under $\p_\Phi$.
The counting measure $\Psi(L_0)$ under $\p_\Phi$ (resp. $\Psi(C_0)$
under $\p_\Phi$)
will be called the {\em point foil} (resp. the
{\em point component}) of $\Phi$ w.r.t. the \pom{} $F$.

The terms point foil and point component are used to stress
that these random counting measures are not always
Palm versions of flow-adapted point processes. 
More precisely, let ${\cal Q}_0$ denote the distribution
of the point foil $\Psi(L_0)$. 
If the foliation of $C_0$ is not markable, then
${\cal Q}_0$ is not the Palm distribution of a flow-adapted
point process (see Subsection \ref{sec:parti}). Similarly, if $C_0$
is not markable, then the distribution ${\cal R}_0$
of $\Psi(C_0)$ is not the
Palm distribution of a stationary
point process.

It follows from the above considerations
that both in the markable and the
non-markable cases, ${\cal Q}_0$ (resp. ${\cal R}_0$) is preserved
by $\theta_{\rfo}$ (resp. $\theta_h$). This invariance property
is of course classical in the markable case.

The fact that it holds in general can be phrased
as follows: for all 
(non--necessarily measure preserving) 
dynamics $f$ on a stationary
point process, there exists
a dynamics $\rfo$ on the typical leaf of the stable manifold of $f$,
which is bijective, dense (has the whole leaf as orbit), 
and which preserves the law of the leaf.
\section{Statistical Properties of \PoMk{} Foils}
\label{sec:quant}
\subsection{Foil Cardinalities}
The following proposition establishes
a connection between the Palm-distribution of $l_n(0)$ (the
cardinality of the set of $\f$-cousins of 0 with the same $n$-th
order ancestor) and the distribution
of $d_n(0)$ (the cardinality of set of $\f$ descendants of generation $n$
w.r.t. $0$ -- see Definition~\ref{foliation}):
\begin{prop}\label{rn}
For all point-shifts $\f$,  for all
stationary point processes $(\Phi,\p)$, 
for all $h:\mathbb N\to \R^{+}$,
\eq{\label{rne}\ep\lk h(d_n(0))\rk=h(0)\fp\lk 0\notin
\f_\Phi^n(\Phi)\rk+\ep\lk \frac{h(l_n(0))}{l_n(0)}\rk.}
\end{prop}
\begin{proof} 
Let   
\eqn{w(\phi,x,y)=\ind\{\rfp^n(x)=y\}\frac {h(l_n(x))}{l_n(x)},}
where $\phi = \Phi(\om)$.
For all $x$ and $y$ in $\phi$,
\eqn{w^+(x)=\frac {h(l_n(x))}{l_n(x)},\quad w^-(y)=
\ind\{d_n(y)\neq 0\}h(d_n(y)),}
and therefore using Lemma \ref{w},
\begin{eqnarray*}
\fe\lk\frac {h(l_n(0))}{l_n(0)}\rk&=&\fe\lk\ind\{d_n(0)\neq 0\}h(d_n(0))\rk\\
&=&\fe\lk h(d_n(0))\rk-\fe\lk \ind\{d_n(0)= 0\}h(d_n(0))\rk\\
&=&\fe\lk h(d_n(0))\rk-\fp\lk d_n(0)= 0\rk h(0)\\
&=&\fe\lk h(d_n(0))\rk-\fp\lk 0\notin \f_\Phi^n(\Phi)\rk h(0).
\end{eqnarray*}
\end{proof}

The announced quantitative results are given in 
the following corollaries of Proposition \ref{rn}. 

If in (\ref{rne}) $h(x)$ is replaced by $xh(x)$, one get: 

\begin{cor} For all $n\ge 0$,
\eq{\label{nre}\ep\lk h(l_n(0))\rk=\ep\lk d_n(0)h(d_n(0))\rk.}
\end{cor}

\begin{cor}
\label{cor25}
For all $n\ge 0$, 
\eq{\label{pev}\fp\lk 0\in \f_\Phi^n(\Phi)\rk =\fe\lk\frac{1}{l_n(0)}\rk.} 
In addition
\eq{\label{pevinf}\fp\lk 0\in \f^\infty(\Phi)\rk =\fe\lk\frac{1}{l_\infty(0)}\rk.} 
\end{cor}
\begin{proof}
The first result is obtained by putting $h\equiv1$ in (\ref{rne}).
Equation (\ref{pevinf}) is obtained
when letting $n\to \infty$ in Equation (\ref{pev}) and when
using monotone convergence.
\end{proof}

Equation (\ref{pevinf}) immediately proves:
\begin{cor}
\label{cor26}
$\f$ evaporates $(\Phi,\p)$ if and only if
the $\f$-foil of $0$ is $\fp$ a.s. infinite\footnote{Equivalently, the
iterated images of $C$, seen as counting measures, converge to
0 for the vague topology.}.
\end{cor}
This is consistent with the
result of Corollary \ref{coreva}
since the property that the foil of $0$ is infinite a.s.
is equivalent to having all connected components of Class~${\mathcal I}/{\mathcal I}$,
or equivalently to having
$(\Phi_{{\mathcal F}/{\mathcal F}},\p)$ and $(\Phi_{{\mathcal I}/{\mathcal F}},\p)$ almost surely empty.
\begin{cor} For all $n\ge 0$,
\eq{ \label{pew}\fe\lk d_n(0)|0\in\rf^n_\Phi(\Phi)\rk=1/\fe\lk\frac{1}{l_n(0)}\rk.}
\label{cor22}
\end{cor}
\begin{proof}
Taking $h$ the identity in (\ref{ftr}) and (\ref{rne}) gives
\begin{eqnarray*}1 = \fe\lk d_n(0)\rk&=&\fe\lk d_n(0)|0\notin\rf^n_\Phi(\Phi)\rk\fp\lk 0\notin\rf^n_\Phi(\Phi)\rk\\
&&+\fe\lk d_n(0)|0\in\rf^n_\Phi(\Phi)\rk\fp\lk 0\in\rf^n_\Phi(\Phi)\rk\\
&=&\fe\lk d_n(0)|0\in\rf^n_\Phi(\Phi)\rk\fp\lk 0\in\rf^n_\Phi(\Phi)\rk.
\end{eqnarray*}
Replacing $\fp\lk 0\in\rf^n_\Phi(\Phi)\rk$ using (\ref{pev}) implies
the result.
\end{proof}

\begin{cor}
\label{cor28}
If $f$ does not evaporate $(\Phi,\p)$, then
\eq{ \label{pewj} 
\fe\lk d_n(0)|0\in\rf^n(\Phi)\rk
\uparrow_{n\to \infty} 1/\fe\lk\frac{1}{l_\infty(0)}\rk<\infty} 
and
\eq{ \label{pewi}\limsup_{n\to \infty} 
\fe\lk d_n(0)|0\in\rf^\infty(\Phi)\rk
\le 1/\fe\lk\frac{1}{l_\infty(0)}\rk<\infty.} 
\end{cor}
\begin{proof}
The first assertion follows from Equation (\ref{pew}).
The second follows from
Equation (\ref{pew}) and simple monotonicity arguments.
\end{proof}

This subsection is concluded with a few observations.
\begin{enumerate}
\item If $\f$ evaporates $\Phi$, namely if $\Phi_{{\mathcal F}/{\mathcal F}}$ and
$\Phi_{{\mathcal I}/{\mathcal F}}$ are empty, then each $\f$-foil
of $\Phi$ has an infinite number of points (Corollary \ref{cor26}),
and the typical point
has a number of descendants which is a.s. finite (Corollary \ref{cor25}) 
but with infinite mean (Proposition \ref{mn1}), and hence heavy tailed.
See Subsection \ref{ssStrip} for an example.

\item
If $\Phi_{{\mathcal I}/{\mathcal I}}$ is empty, then each $\f$-foil
of $\Phi$ has an a.s. finite number of points 
(Theorem \ref{cardinality.components});
the typical point has descendants of all orders with a positive probability\footnote{With probability 1 iff $\f$ is bijective.}
(Corollary \ref{cor25}),
and hence an infinite number of descendants. However,
the expected number of descendants of order $n$ 
does not diverge in mean (Corollary \ref{cor28}) as $n$ tends to infinity.
If in addition $\Phi_{{\mathcal I}/{\mathcal F}}$ is empty, then
the set of descendants of the typical point
looks like a ``finite star with a loop at the center''.
See Subsection \ref{ssMnn} for an example.
If in place $\Phi_{{\mathcal F}/{\mathcal F}}$ is empty, then
the set of descendants of the typical point is either finite or
looks like an ``infinite path with finite trees attached to it''.
The points in this infinite path
constitute a sub-stationary point process.
This point process always has a
positive intensity.
(For instance, for the RLS \pom{} on the Poisson
point process in $\R^2$, this is the
whole point process. There exist cases where 
$F$ is not bijective and the connected components are 
of type $\cal I/F$ and hence such that this 
sub-point process is not the whole point process.
\end{enumerate}

\subsection{Foil Intensities} 
This subsection is focused on the intensity of the $\f$-foils.
From Proposition \ref{thdic}, either all foils of a markable component are 
flow-adapted point processes, or none of them are.
The notion of intensity only makes sense in the former case.
The notion of {\em relative intensity} defined in the next subsection allows
one to discuss the ``density'' of foils in whole generality, namely
regardless of the above dichotomy.

\subsubsection{Relative Intensities}
Below, when considering a component
of Class~${\mathcal I}/{\mathcal I}$, it is assumed that
$\fo$ is an $\L[\Phi]\f$-dense 
element of the $\f$-stable group and that
$\preceq_{\fo}$ is $\f$-compatible.

Let $\g_\bot$ denote the \pomk{} of $\fo$ and let $\theog$
denote its related shift on $\n^0$. Equations
(\ref{fgfol}) and (\ref{orthogonal.distance}) give that for $\fp$-almost
all $\phi\in\n^0$, 
\eq{\label{orthogonal.fg}\theog^n\phi=\theta_{\fo^n(\phi,0)}\phi.}
Hence if $\phi\sim_\thetg \psi$,  with abuse of notation,
one can define $\dd(\phi,\psi)$ as the unique integer $n$ such that
$\theog^n\phi=\psi$.

Consider the dynamical system $(\n^0,\theta_{\rfo})$.
The fact that $\fo$ is bijective implies that $\theta_{\rfo}$
preserves $\fp$. 

\begin{thm}
\label{existence.of.relative.intensity}
Let $(\Phi,\p)$ be a stationary point process, $\f$
be an arbitrary \pom{} and  $\fo$ and $\dd$ 
be as in Definition~\ref{f.orthogonal}.
Then, for $\cp_0$ almost-all realizations $\phi$, the limit 
\eq{\label{relative.intensity}\lambda_+(\phi):=
\lim_{n\to\infty}
\frac{\dd\lp\fph(0),\fph\circ\fo^n(0)\rp}{\dd\lp 0,\fo^n(0)\rp}}
exists, is positive and
in $L_1({\cal P}_0)$.
In addition, $\lambda_+(\phi)$
is a function of the foil of $0$ only; i.e., if
$0\sim_\f x$,  $\lambda_+(\theta_x\phi)=\lambda_+(\phi)$ and 
$\lambda_+$ is  independent
of the choice $\fo$ as far as it satisfies
the properties in Proposition~\ref{existence.f.orthogonal}.
\end{thm}

\begin{rem}
The existence of the non-degenerate limit in (\ref{relative.intensity})
can be seen as a proof of the fact that all foils of a connected
components have the ``same dimension''. This fact justifies the
use of the term ``foliation'' within this context (see e.g. \cite{An01}).
\end{rem}
\begin{proof}
If $x$ is in a connected component of $G^\f(\phi)$ which is 
${\mathcal F}/{\mathcal F}$ or
${\mathcal I}/{\mathcal F}$, all statements 
follow from finiteness of the foils.
Hence assume $C(0)$ is ${\mathcal I}/{\mathcal I}$. 
Hence it is sufficient to show that, for $\cp_0$ almost
all $\phi$, the  limit  
\eq{\label{relative.intensity.n0}\lambda_+(\phi)=
\lim_{n\to\infty}\frac{\dd\lp\thetg \phi,\thetg\circ\theog^n\phi\rp}
{\dd\lp \phi,\theog^n\phi\rp},}
exists and is positive, finite and constant on the
foil $\l\thetg(\phi)$ provided the latter is infinite.
Let $ \dd_+(\phi):= \dd(\thetg\phi,\thetg\circ\theog\phi)$.
Now consider the following mass transport:
\eqn{w(\phi,x,y)=
\begin{cases}
1 & y\in \l[+]\f(x) \text{ and } 0  {\le}
\dd(\f_\phi(x),y)
 {<}
\dd(\f_\phi(x), \f_\phi\circ\fo(x))\\
0 & \text{otherwise}.
\end{cases}
}
One has, for all points $x,y\in\phi$, $w^-(y)\leq 1$
and $w_+(x)= \dd_+ (\theta_x\phi)$. Therefore
\eqn{\fe\lk \dd_{+}\rk = \fe\lk w^+(0)\rk = \fe\lk w^-(0)\rk\leq 1.} 
Since  the denominator in (\ref{relative.intensity.n0}) is equal to $n$, 
one has 
\begin{eqnarray*}
\lim_{n\to\infty}\frac{\dd\lp\thetg\phi,\thetg\circ\theog^n\phi\rp}{\dd\lp \phi,\theog^n\phi\rp}
&=&\lim_{n\to\infty}\frac 1n {\sum_{i=1}^n \dd\lp\thetg\circ\theog^{i-1}\phi,\thetg\circ\theog\circ\theog^{i-1}\phi\rp}\\
&=& \lim_{n\to\infty}\frac 1n {\sum_{i=0}^{n-1}  \dd_+ \lp\theog^{i}\phi\rp}.
\end{eqnarray*}
Since $\fo$ is $\p$-\asy{} bijective, $\fp$ is $\theog$-invariant.
Therefore if one denotes by $\cal I$
 the invariant $\sigma$-field of $\theog$,
  the finiteness of $\fe[ \dd_+]$ implies that the last
limit exists for $\fp$-almost all $\phi$ and
it is equal to $\fe[ \dd_+|{\cal I}]$, which is finite
and invariant under the action of $\theog$;
i.e., it is a function of $\l\thetg(\phi)$. 

To prove that $\lambda_+(\phi)$ is a.s. positive, note that, if $Y$
is the event of being the youngest son of the family, then
$\tilde \dd\geq \ind_Y$. Hence if, with positive probability, 
$\lambda_+(\phi)=0$, this means that, with positive probability,
$\fe[\ind_Y|{\cal I}]$ is zero. But since $\fe$ is $\theog$-invariant,
this means that, with positive probability, there is no youngest son
on the foil of $\l\thetg(\phi)$, which means that all points of
$\l\thetg(\phi)$ are brothers. 
Since $C^\thetg(\phi)$ is infinite, this contradicts the a.s.
finiteness of $d_1(\f_\phi(0))$.

Finally to prove that $\lambda_+(\phi)$ is independent of the choice of $\fo$,
it is sufficient to show that $\fe[{\Delta_+}|{\cal I}]$ depends only on $f$.
To do so, it is enough to prove that
for all $A\in {\cal I}$, $\fe[{\Delta_+}\ind_{A}]$ depends only on $f$.
Since $A\in{\cal I}$, with abuse of notation, one has 
$\ind_A(\phi) = \ind_A(\l\thetg(\phi))$.
Let
\begin{eqnarray*}
A_+ &=&\{\thetg \phi; \phi\in A\},\\
\l[-]\thetg(A) &=& \{\phi\in\n^0;\l[-]\thetg(\phi)\neq\emptyset \text{ and } \l[-]\thetg(\phi)\in A \}.
\end{eqnarray*}
It is easy to see that $\l[-]\thetg(A)\in{\cal I}$. Let  
$$u_{\phi}(x,y)={\Delta_+}(\theta_x \phi)\ind_{\{y=F_\phi(x)\}}\ind_{A}(\theta_x \phi),$$ 
which is a flow-adapted transport kernel. If $A_+$ denotes $\{\thetg \phi; \phi\in A\}$, by the mass transport principle,
\eq{\label{expecDelta}
\fe\lk{{\Delta_+}\ind_{A}}\rk=\fe\lk{w^+(0)}\rk= \fe\lk{w^-(0)}\rk = \fe\lk{{\Delta_\bot}\ind_{A_+}}\rk,}
where ${\Delta_\bot}(\phi)$ is the smallest $i>0$ such that $F_{\perp}^i(x)$ has a child and zero otherwise. Note that all elements of $A_+$ have at least one child and therefore $\ind_{A_+}(\phi)$ is zero whenever $\phi$ has no child.  
Let
\[
v_{\phi}(x,y)=\left\{\begin{array}{ll}
1 & \theta_x(\phi)\in A_+ \text{ and } y=f_{\bot}^i(x) \text{ for some } 0\leq i <{\Delta_\bot}(x),\\
0 & \text{ otherwise}.
\end{array}\right.
\]
Since $A\in{\cal I}$, (\ref{expecDelta}) and the mass transport principle give
\begin{eqnarray*}
\fe\lk{{\Delta_+}\ind_{A}}\rk&=&\fe\lk{{\Delta_\bot}\ind_{A_+}}\rk = \fe\lk{v^+(0)}\rk\\
&=&\fe\lk{v^-(0)}\rk = \fe\lk\ind_{\l[-]\thetg(A)}\rk = \fp\lk\l[-]\thetg(A)\rk.
\end{eqnarray*}
Clearly the latter depends only on $f$ and not on the choice of $F_\bot$ which completes the proof. 
\end{proof}

\begin{cor}
Letting $A=\Omega$ in the last proof gives  
$$\fe[\Delta_+]=\fp[{\l[-]\thetg(\Omega)}],$$ 
where the R.H.S. is the probability that 0 is not in the first
foil (if there is any) of its component.
\end{cor}

\begin{defi}
The quantity $\lambda_+(\Phi)$, defined $\p_\Phi$ a.s.,
counts the average number of different points in the foil
$\l[+]\f(0)$ per point in the foil of 0, $\l \f(0)$,
and is hence called the \emph{relative intensity of $L_+^\f(0)$
with respect to $L^\f(0)$} in $\Phi$.
This notion extends to the relative intensity
$$\Lambda_+(x,\Phi)=\lambda_+(\theta_x \Phi)$$
of $L_+^\f(x)$ with respect to $L^\f(x)$ for all $x\in \Phi$.
\end{defi}

\subsubsection{Intensities}
In the particular case where foils are markable,
one gets back the following classical result
as a direct corollary of Theorem \ref{existence.of.relative.intensity}:

\begin{prop}
\label{prop-same-int}
Assume that $(\p,\theta_t)$ is ergodic.
Assume $L^\f(0)$ is markable, so that it is the support of a point process.
Let $\beta$ (resp. $\beta_+$) denote the intensity
of $\Psi(L^\f(0))$ (resp. $\Psi(L_+^\f(0))$). Then
$\beta_+=\beta \Lambda_+,$
where $\Lambda_+={\mathbb E}_0(\Lambda_+(0,\Phi)).$ 
\end{prop}
Note that it follows from $\beta_+=\beta \Lambda_+$
that $\Lambda_+<1$.

\section*{Acknowledgements}
The authors thank Ali Khezeli for his helpful comments and his 
contribution to the proof of Theorem \ref{existence.of.relative.intensity}.
They also thank Antonio Sodre for his useful
comments on this paper. 

\section{Appendix}
\subsection{Mass Transport}
Let $w$ be a point-shift and $(\Phi,\p)$
be a stationary point process. Let $G^w(\Phi)$
be the directed graph of Definition \ref{def-graph}.
Define $w^+(\Phi, 0)$
(ref. $w^-(\Phi, 0)$) to be the out-degree
(in-degree) of node $0$ 
under $\p_\Phi$.
The following is classical:

\begin{lem}[Mass transport principle]\label{w} 
If $w$ is a mass transport and $(\Phi,\p)$
is a stationary point process then 
\eq{\label{wtr}\fe\lk w^+(\Phi, 0)\rk=\fe\lk w^-(\Phi, 0)\rk.}
\end{lem}
\subsection{Depth First Search}
\label{AB}
DFS is a recursive 
algorithm prescribing a class of ways to traverse
a rooted tree.
Nodes belong to two categories: visited and unvisited.
The algorithm starts from the root, with the latter visited
and all other nodes unvisited.
From a given visited node,
the node visited next is one of its yet unvisited sons.
If all its sons have already been visited (in particular if it 
has no sons), then the algorithm moves to the father of the given node
to search for the next unvisited node.

\subsection{Multi Type Strip  \PoM{}}
\label{ssMultiStrip}
Consider the following variant of the Strip \PoM{}.
To each point $x_i$ of the Poisson point process $\Phi$, one
associates an independent mark $m_i$, which is a Poisson 
point process of intensity 1 on a circle of radius 1. 
Consider the (Poisson cluster) point process
$$\Psi=\Phi +\sum_{i} x_i+m_i.$$
Each realization of $\Psi$
determines the points $x_i$ of $\Phi$ and the associated
cluster $x_i+m_i$. It hence allows one to classify
the points of $\Phi$ in types taking their values in $\mathbb N$,
with the type of $x_i$ being the cardinality of $m_i$.
The Multi Type Strip \PoM{} $f$
maps $y\in x_i+m_i$ to $x_i$ and uses
the $f_s$ map within points of type $k\in \mathbb N$ with $\Phi$.

On $\Psi$, this \pom{} admits an infinite number of connected
components (one per type). It follows from
the results of Subsection \ref{ssStrip} (and from
the fact that the points of type $k$ in $\Phi$
form a stationary Poisson point process of positive intensity
that each connected component has properties similar that that
of the unique component of Subsection \ref{ssStrip}; in
particular, it is
of Class~${\mathcal I}/{\mathcal I}$ and evaporates under the action of $f$.

\bibliographystyle{amsplain}
\bibliography{bib-01-16}

\end{document}